\documentclass[a4paper,12pt]{amsart}

\usepackage{amssymb} 
\usepackage{xcolor} 
\usepackage{fancybox}
\usepackage{dashbox}
\newboolean{showcomments}
\setboolean{showcomments}{true}
\ifthenelse{\boolean{showcomments}}
{ \newcommand{\mynote}[3]{
  \fbox{\bfseries\sffamily\scriptsize#1}
  {\small$\blacktriangleright$\textsf{\emph{\color{#3}{#2}}}$\blacktriangleleft$}}}
{ \newcommand{\mynote}[3]{}}
\definecolor{asparagus}{rgb}{0.53, 0.66, 0.42}

\usepackage{mathrsfs}

  \usepackage{float}

  \usepackage{pgffor}
  \usepackage{tikz}
  \usetikzlibrary{calc,3d,arrows,positioning,shapes.misc}

  \usepackage[english]{babel}
  \usepackage{amsthm}
  \usepackage{amsmath}
  \usepackage{amssymb}
  \usepackage[latin1]{inputenc}
  \usepackage{setspace}
  \usepackage{enumerate}
  \usepackage{stmaryrd}

  \usepackage[colorlinks=true,linkcolor=black,citecolor=black,filecolor=black,urlcolor=black,menucolor=black]{hyperref}

  \usepackage{array}
  \usepackage{graphicx}
  \usepackage[babel]{csquotes}
  \usepackage{geometry}
  \usepackage{bbm}
  \usepackage{stmaryrd}
  \usepackage[all]{xy}
  \usepackage{mathrsfs}

  \theoremstyle{plain}
  \newtheorem{theorem}{Theorem}
  
  \newtheorem{corollary}{Corollary}
  
  \newtheorem{lemma}{Lemma}
  
  \theoremstyle{definition}
  \newtheorem{definition}{Definition}
  \newtheorem{expl}{Example}
  \newtheorem{rem}{Remark}

	\allowdisplaybreaks[3]

  \def \D{\mathcal{D}}
  \def \E{\mathcal{E}}
  \def \F{\mathcal{F}}
  \def \H{\mathcal{H}}

  \def \R{\mathbb{R}}
  
  \def \transpose{\mathsf{T}}

  \newcommand {\sign} {\mathop \mathrm{sign}}

  \newcommand {\supp} {\mathop \textup{supp}}
  \def \chara{\mathbf{1}}

 \newcommand {\dist} { \mathop \textup{dist} \nolimits}
\newcommand{\p}{\partial}
\newcommand{\dL}{\,\mathrm{d}}
\newcommand{\ddt}{\frac{\mathrm{d}}{\mathrm{d}t}}

  \title[Weak-strong uniqueness for volume-preserving mean curvature flow]{Weak-strong uniqueness for volume-preserving \\mean curvature flow}

\author{Tim Laux}
\address{Tim Laux, Hausdorff Center for Mathematics, University of Bonn,  Villa Maria, Endenicher Alllee 62, 53115 Bonn, Germany}
\email{tim.laux@hcm.uni-bonn.de}

\begin{document}

    \begin{abstract}
	In this note, we derive a stability and weak-strong uniqueness principle for volume-preserving mean curvature flow. 
	The proof is based on a new notion of volume-preserving gradient flow calibrations, which is a natural extension of the concept in the case without volume preservation recently introduced by Fischer et al.~[\href{https://arxiv.org/abs/arXiv:2003.05478}{arXiv:2003.05478}]. 
	The first main result shows that any strong solution with certain regularity is calibrated. 
	The second main result consists of a stability estimate in terms of a relative entropy, which is valid in the class of distributional solutions to volume-preserving mean curvature flow.
		
	\medskip
    

  \medskip

  \end{abstract}
\maketitle



\section{Introduction}

	Volume-preserving mean curvature flow is the most basic geometric evolution equation for closed hypersurfaces that preserves the enclosed volume. 
	More precisely, the equation reads
	\begin{align}\label{eq:introVPMCF}
	V=-H+\lambda \quad \text{on }\Sigma(t),
	\end{align}
	where $V$ and $H$ denote the normal velocity and the mean curvature of the evolving surface $\Sigma(t)=\p \Omega(t)$, respectively, and 
	\begin{align}\label{eq:introlambda}
		\lambda = \lambda(t) := \frac1{\H^{d-1}(\Sigma(t))} \int_{\Sigma^(t)} H \dL \H^{d-1}
	\end{align} 
	is the Lagrange-multiplier corresponding to the volume constraint $|\Omega(t)|=|\Omega(0)| = m$. 
	This system has a gradient-flow structure as is seen at the energy dissipation relation for the area functional $E[\Sigma]=\H^{d-1}(\Sigma)$
	\begin{align*}
	\ddt E[\Sigma(t)]
	= \int_{\Sigma(t)} VH\dL  \H^{d-1}
	=- \int_{\Sigma(t)} V^2 \dL \H^{d-1},
	\end{align*}
	which holds for sufficiently regular solutions of~\eqref{eq:introVPMCF}--\eqref{eq:introlambda}. 
	Precisely,~\eqref{eq:introVPMCF}--\eqref{eq:introlambda} is the $L^2$-gradient flow of the area functional restricted to the ``manifold'' $\{\Sigma=\partial \Omega \subset \R^d \colon |\Omega| = m\}$ which encodes the volume constraint. 
	The equation arises as the singular limit of the nonlocal Allen--Cahn equation by Rubinstein and Sternberg~\cite{RubinsteinSternberg} and is a common model for coarsening processes in which the phase volume is preserved. 
	
	\medskip
	
	Gage~\cite{GageArea} and Escher and Simonett~\cite{EscherSimonett} established the existence of classical solutions for volume-preserving mean curvature flow for short time in two respectively, higher dimensions.
	However, singularities may appear in finite time, even in the case of planar curves~\cite{MayerSimonett}. 
	To describe the evolution through these singular events, several notions of weak solutions have been considered. 
	Mugnai, Seis and Spadaro~\cite{MugnaiSeisSpadaro} constructed solutions based on an energy-convergence assumption using an (almost) volume-preserving version of the scheme considered by Luckhaus and Sturzenhecker~\cite{LucStu}. 
	Swartz and the author~\cite{LauxSwartz} proved the convergence of the volume-preserving thresholding scheme, an efficient numerical algorithm, under a similar condition.
	The latter result also applies to certain multiphase systems with a volume constraint. 
	Also for the nonlocal Allen--Cahn equation~\cite{RubinsteinSternberg}, such a convergence result can be derived, see the work of Simon and the author~\cite{LauxSimon}. 
	In fact, this result applies to any number of phases any selection of which may carry a volume constraint.
	Volume-preserving mean curvature flow can also be formulated for evolving varifolds by extending Brakke's notion~\cite{brakke} of mean curvature flow to this volume-preserving case.
	Takasao~\cite{Takasao} showed that solutions to a slightly modified version of the nonlocal Allen--Cahn equation due to Golovaty~\cite{Golovaty} converges to this varifold solution.
	Recently, Takasao~\cite{TakasaoHigherD} refined his methods by slightly relaxing the volume constraint in the approximation and only recovering the precise volume preservation in the sharp-interface limit, which in particular allowed him to extend his earlier result~\cite{Takasao} to higher dimensions.
	The idea of relaxing the volume constraint in the approximation is in some sense inspired by~\cite{MugnaiSeisSpadaro}. 
	Although volume-preserving mean curvature flow does not obey a naive comparison principle, there is also a way to make the powerful machinery of viscosity solutions work in the case of volume-preserving mean curvature flow by fixing the Lagrange-multiplier for competitors as was shown by Kim and Kwon~\cite{KimKwon}.
	
	\medskip
	
	In this note, we want to address the consistency of weak solutions from~\cite{LauxSimon,LauxSwartz,MugnaiSeisSpadaro} with classical solutions. 
	A priori, it is not evident that these weak solutions agree with the unique strong solution (as long as the latter exists). 
	To draw this connection between these solution concepts, we extend the notion of gradient-flow calibrations introduced in the recent work by Fischer, Hensel, Simon and the author~\cite{FHLS} to the volume-preserving case 
	and show that any sufficiently regular classical solution is calibrated in this sense, see Theorem~\ref{thm:exGFC}.
	Then, in Theorem~\ref{thm:ws}, we show that every calibrated flow is unique and stable in the class of distributional solutions.
	The proofs are self-contained and elementary. 
	The main novelty of this work is a suitable extension $B$ of the velocity field in the definition of gradient-flow calibrations. 
	Instead of an ad-hoc extension by nearest-point projection onto the classical solution, we solve a Neumann--Laplace equation to guarantee that next to the usual conditions, $B$ also satisfies the incompressibility condition $\nabla \cdot B=0$, at least with a linear error as one moves away from the interface.
	Surprisingly, with this construction, no additional estimate on the closeness of the respective Lagrange-multipliers is needed to derive the relative entropy inequality.
	
	\medskip
	
	The relative entropy method and the notion of gradient-flow calibrations in~\cite{FHLS} has led to several recent results for geometric evolution equations. 
	The method can be used to prove quantitative convergence of the Allen--Cahn equation to mean curvature flow as was shown by Fischer, Simon and the author~\cite{FischerLauxSimon}.
	One of the main advantages of the method is its simplicity and its applicability in vectorial problems as it does not require a spectral analysis of the linearized Allen--Cahn operator and is not based on the comparison principle. 
	Liu and the author~\cite{LauxLiu} combined the relative entropy method with weak convergence methods to derive the sharp-interface dynamics of isotropic-nematic phase transitions in liquid crystals.
	Most recently, Fischer and Marveggio~\cite{FischerMarveggio} extended the result~\cite{FischerLauxSimon} to the vector-valued Allen--Cahn equation and proved its convergence to multiphase mean curvature flow. 
	Previous to this result, only formal arguments~\cite{BronsardReitich} and the conditional result~\cite{LauxSimon} were known.
	One can also lower the assumptions on the weak solution to the bare minumum of a suitable optimal energy-dissipation relation as was shown by Hensel and the author~\cite{HenselLauxVarifold}, which underlines the importance of the underlying gradient-flow structure of~\eqref{eq:introVPMCF}--\eqref{eq:introlambda}. 
	Also boundary conditions can be naturally incorporated in the method as was shown by 
	Hensel and Moser~\cite{HenselMoserContact}, and Hensel and the author~\cite{HenselLauxContact}.
	
	\medskip
	
	We expect that also in this volume-preserving version, the method will be a useful tool for further work, such as quantitative convergence results for phase-field models in the sharp-interface limit or the analysis of the long-time behavior of solutions. 
	The former has been done in a qualitative way in the previously mentioned works~\cite{LauxSwartz, Takasao, LauxSimon, TakasaoHigherD}.
	The latter problem has been addressed with different methods in~\cite{GageArea, HuiskenVol, EscherSimonett,JulinMoriniPonsiglioneSpadaro, DeGennaroKubin}.
	Another interesting possible future application of these methods is a local minimality criterion for constant mean curvature hypersurfaces with respect to volume-preserving distortions.
	We mention that the methods developed here should naturally extend to the setting of varifold solutions like the ones constructed by Takasao~\cite{Takasao, TakasaoHigherD}.
	In the unconstrained case of standard mean curvature flow, this has been shown in~\cite{HenselLauxVarifold}.
	Also an extension to sufficiently regular, strongly convex anisotropies in both surface tension and mobility function seems feasible. 
	In the case of standard mean curvature flow, this is part of the master's thesis~\cite{thesis_clemens_ullrich}, which also extends ideas from~\cite{Laux_anisotropic}.
	However, in the case of non-smooth or non-strongly convex anisotropies, this has not yet been explored.
	Finally, let us mention the recent work~\cite{JulinNiinikoski}, in which with completely different methods, it is shown that the scheme~\cite{MugnaiSeisSpadaro} converges to volume-preserving mean curvature flow before the onset of singularities.
		
	The remainder of this paper is organized as follows. 
	In Section~\ref{sec:main}, we state the main definitions and results.
	In Section~\ref{sec:exGFC} we construct the gradient-flow calibrations to prove Theorem~\ref{thm:exGFC}.
	Finally, in Section~\ref{sec:ws}, we prove Theorem~\ref{thm:ws} by deriving a relative entropy inequality which allows to close a Gronwall argument.

	We will use the following notation throughout. 
	We write $a\lesssim b$ if there exists a constant $C<\infty$ depending on $d$, $T^*$, and $\Sigma^*=(\Sigma^*(t))_{t\in[0,T^*]}$, such that $a\leq Cb$. 
	The Landau symbol $O$ will be used frequently. Precisely, by $a=O(b)$ we mean that there exists a constant $C<\infty$ depending on $d$, $T^*$, and $\Sigma^*=(\Sigma^*(t))_{t\in[0,T^*]}$ defined below, such that $|a| \leq C |b|$.
		
	\section{Main results}\label{sec:main}
	
	Let us first define the notion of gradient flow calibrations in the context of volume-preserving mean curvature flow. 
	\begin{definition}\label{def:GF_Cal}
		Let $\Sigma^*=(\Sigma^*(t))_{t\in[0,T^*]}$ be a one-parameter family of closed surfaces $\Sigma^*(t) = \partial \Omega^*(t) \subset \R^d$. 
		Let $\xi,B \colon \R^d\times[0,T^*]\to \R^d$, let $\vartheta \colon \R^d\times[0,T^*]\to \R$, and let $ \lambda^* \colon [0,T^*]\to \R$.
		We call the tuple $(\xi,B, \vartheta, \lambda^*)$ a \emph{gradient-flow calibration for volume-preserving mean curvature flow} if the following statements hold true.
		\begin{enumerate}[(i)]
			\item \emph{Regularity.}\label{item:reg}
			The vector field $\xi$ and the functions $\vartheta$ satisfy
			\begin{align*}
				\xi \in 
				C^{1}_c(\R^d\times[0,T^*];\R^d) \quad \text{and} \quad   \vartheta \in C^{0,1}(\R^d\times[0,T^*]) \cap L^\infty(\R^d\times[0,T^*]).
			\end{align*}
			Furthermore, for each $t\in[0,T^*]$ it holds
			\begin{align*}
				B(\cdot,t) \in C^{1,1}(\R^d;\R^d).
			\end{align*}
			\item \emph{Vanishing Divergence.} \label{item:divB}
			The vector field $B$ satisfies for each $t\in [0,T^*]$
			\begin{align}\label{eq:divB}
				\nabla \cdot B (\cdot,t)=O\big(\dist(\cdot, \Sigma^*(t))\big).
			\end{align}
			\item \emph{Normal extension and shortness.}\label{item:normal}
			The vector field $\xi$ extends the exterior unit normal vector field of $\Sigma^*$, i.e., 
			\begin{align}\label{eq:xiext}
				\xi(\cdot,t) = \nu^*(\cdot,t) \quad \text{on } \Sigma^*(t),
			\end{align}
			and it is short away from $\Sigma^*$ in the sense that there exists a constant $c>0$ such that
			\begin{align}\label{eq:xishort}
				|\xi(\cdot,t)| \leq \max\big\{(1-c\dist^2(x,\Sigma^*(t)), 0\big\},
			\end{align}
			\item \emph{Approximate transport equations.}\label{item:transport}
			They weight $\vartheta$ is transported to first order
			\begin{align}\label{eq:transp_weight}
					\left(\p_t \vartheta + (B\cdot \nabla)\vartheta \right)(\cdot,t) = O\big(\dist(\cdot,\Sigma^*(t))\big),
			\end{align}
			and the length of $\xi$ to second order
			\begin{align}\label{eq:transp_absxi}
				\left(\p_t |\xi|^2 + (B\cdot \nabla) |\xi|^2\right)(\cdot,t) = O\big(\dist^2(\cdot,\Sigma^*(t))\big).
			\end{align}
			Furthermore, there exists a constant $C<\infty$ and a function $f\colon \R^d\times[0,T^*]\to \R$ with $\|f(\cdot,t)\|_{L^\infty} \leq C$ for all $t\in[0,T^*]$  such that the vector field $\xi$ is almost transported by $B$ in the sense that 
			\begin{align}\label{eq:transp_xi}
				\left(\p_t \xi + (B\cdot \nabla ) \xi + (\nabla B)^{\transpose} \xi \right)(\cdot,t)
				= f(\cdot,t)\,\xi(\cdot,t)+ O\big(\dist(\cdot,\Sigma^*(t))\big).
			\end{align}

			\item \emph{Geometric evolution equation.}\label{item:GEE}
			It holds
			\begin{align}\label{eq:extGEE}
				B(\cdot,t)\cdot \xi(\cdot,t) +\nabla \cdot \xi(\cdot,t) -\lambda^* (t)= O\big(\dist(\cdot,\Sigma^*(t))\big)
			\end{align}
			and the function $\lambda^*\colon[0,T^*]\to \R $ is given by
			\begin{align}\label{eq:deflambda_cali}
					\lambda^*(t) := \frac1{\H^{d-1}(\Sigma^*(t))} \int_{\Sigma^*(t)} \nabla \cdot \xi(\cdot,t) \dL \H^{d-1}.
			\end{align}
			
			\item \emph{Sign condition on and coercivity of transported weight.} \label{item:signweights}
			We have
			\begin{align*}
				\vartheta(\cdot, t) &<0 \quad \text{in }\Omega^*(t),
				\\ \vartheta(\cdot, t) &>0 \quad \text{in }\R^{d} \setminus \overline{\Omega^*(t)}.
			\end{align*}
			Furthermore, there exists a constant $c>0$ such that
			\begin{align}\label{eq:thetacoercive}
				\min\{\dist(\cdot, \Sigma^*(t)) ,c\} \leq |\vartheta(\cdot,t)|.
			\end{align}
			
		\end{enumerate}
	In case such a gradient-flow calibration exists for $\Sigma^*$, we call $\Sigma^*$ a \emph{calibrated flow}.
	\end{definition}
	
	All the quantities $\xi, B, \vartheta, \lambda^*$ in the definition have natural interpretations. 
	First, $\xi$ is an extension of the normal vector field $\nu^*$. 
	Second, $B$ is an extension of the velocity vector field $V^*\nu^*$ with unprescribed tangential part but with the additional property that it is solenoidal, which is compatible with the volume-preservation of the PDE~\eqref{eq:VPMCF}. 
	Third, $\vartheta$ is a suitably truncated version of the signed distance function to $\Sigma^*(t)$. 
	Lastly, $\lambda^*=\lambda^*(t)$ corresponds precisely to the Lagrange-multiplier~\eqref{eq:introlambda} appearing in the PDE~\eqref{eq:introVPMCF}.
	
	Note carefully, that the extended velocity vector field $B(\cdot,t)$ does not need to point in normal direction on $\Sigma^*(t)$. 
	In fact, as will be seen in our construction, in general $B(\cdot,t)$ will have  a nontrivial tangential component, which is of course compatible with the geometric invariance of the evolution equation~\eqref{eq:introVPMCF}--\eqref{eq:introlambda}. 
	
	On a technical note, it is interesting that we do not need to impose any assumption on the dependence of $B$ on the time variable $t$. 
	The map $t\mapsto B(\cdot,t)$ does not have to be measurable, let alone continuous in any sense.
	
	\medskip
	
	The first main result states that every classical solution to volume-preserving mean curvature flow (with some regularity assumption stated in  Definition~\ref{def:strong} below) is calibrated in the sense of Definition~\ref{def:GF_Cal}.
	
	\begin{theorem}\label{thm:exGFC}
			Let $\Sigma^*=(\Sigma^*(t))_{t\in[0,T^*]}$ be a regular solution to volume-preserving mean curvature flow in the sense of Definition~\ref{def:strong} below.
			Then there exists a gradient-flow calibration $(\xi, B, \vartheta, \lambda^*)$ of $\Sigma^*$.
	\end{theorem}

	\begin{definition}\label{def:strong}
		Let $\Sigma^* = (\Sigma^*(t))_{t\in[0,T^*)}$, with $\Sigma^*(t)=\partial \Omega^*(t)$ and $\Omega^*(t)$ bounded, say, $\Omega^*(t) \subset B_{R^*}(0)$ for all $t\in[0,T^*]$. 
		Then we call $\Sigma^*(t)$ a \emph{regular solution of volume-preserving mean curvature flow} if
		$\Sigma^*(t)$ is of class $C^{3,\alpha}$ with normal velocity field $V^*$ of class $C^{2,\alpha}$, and for all $t \in [0,T^*]$ it holds
		\begin{align}\label{eq:VPMCF}
			V^* = -H^* +\lambda^*\quad \text{on } \Sigma^*(t),
		\end{align}
		where $\lambda^* = \lambda^*(t)$ is the Lagrange-multiplier corresponding to the volume-constraint $|\Omega(t)| = |\Omega(0)| =:m^*$, which is explicitly given by
		\begin{align}\label{eq:lambda*}
			\lambda^* (t) = \frac1{\H^{d-1}(\Sigma^*(t))} \int_{\Sigma^*(t)} H^*\dL \H^{d-1}.
		\end{align}
	\end{definition}
	
	We now want to state the precise definition of distributional solution to volume-preserving mean curvature flow used in this work. 
	To this end, let us introduce some notation from the theory of functions of bounded variation and sets of finite perimeter. 
	We use the (standard) notation 
	\begin{align*}
		E[\chi(\cdot,t)] := \int_{\R^d}| \nabla \chi(\cdot,t)| =\sup \bigg\{ \int_{\R^d} (\nabla \cdot \xi) \chi(\cdot,t) \dL x \colon \xi\in C^1(\R^d;\R^d), \, |\xi| \leq 1\text{ in } \R^d\bigg\}
	\end{align*} 
	to denote the total mass of the time-slice of the total variation measure $|\nabla \chi|$, which corresponds to the perimeter of the set $\{ \chi(\cdot,t)=1\}$. 
	Furthermore, we denote the (measure-theoretic) exterior normal to the set of finite perimeter $\{\chi(\cdot,t)=1\} \subset \R^d$ by $\nu(\cdot,t) =- \frac{\nabla \chi(\cdot,t) } {|\nabla \chi(\cdot,t)|}$, which satisfies $\nabla \chi(\cdot,t) = - \nu(\cdot,t) |\nabla \chi(\cdot,t)|$.
	Now we are in the position to state the definition of distributional solutions.
	
	\begin{definition}\label{def:weak}
		A measurable function $\chi \colon \R^d\times(0,T) \to \{0,1\} $ is called a \emph{distributional solution to volume-preserving mean curvature flow} if there exists a $|\nabla \chi|$-measurable function $V \colon \R^d \times (0,T)\to \R$ and a measurable function $\lambda\colon(0,T)\to \R$ such that the following statements hold. 
		\begin{enumerate}[(i)]
			\item  \label{item:V} \emph{Normal velocity.}
			For all test functions $\zeta\in C^1(\R^d\times[0,T))$ and almost every $T'\in (0,T)$ it holds
			\begin{align}
			\notag\int_{\R^d} \zeta(\cdot,T') \chi(\cdot,t)\dL x
			- \int_{\R^d} & \zeta(\cdot,0)\chi(\cdot,0) \dL x 
			\\ \label{eq:defV} &=
			\int_{\R^d\times(0,T')} \chi \p_t \zeta \dL x \dL t +\int_{\R^d\times(0,T')} \zeta V  |\nabla \chi| \dL t.
			\end{align}
			\item \label{item:weakevolutioneq} \emph{Evolution equation.}
			For all test vector fields $B\in C^1(\R^d ;\R^d)$ and almost every $t\in (0,T)$ it holds
			\begin{align}\label{eq:weakevolutioneq}
				\int_{\R^d\times\{t\}} \big(\nabla \cdot B-\nu \cdot \nabla B \, \nu \big) \, |\nabla \chi|
				= -\int_{\R^d\times\{t\}}  (V-\lambda)\nu \cdot B \,|\nabla \chi|.
			\end{align}
			
			\item\label{item:EDI} \emph{Optimal energy dissipation rate.}		
			For almost every $T'\in (0,T)$ we have
			\begin{align}\label{eq:EDI}
				E[\chi(\cdot,T')] +  \int_{\R^d\times(0,T')}  V^2  |\nabla \chi| \dL t \leq E[\chi(\cdot,0)].
			\end{align}
			
			\item \label{item:VolPres} \emph{Volume preservation.} 
			For almost every $t\in(0,T)$ 
			\begin{align}\label{eq:VolPres}
				\int_{\R^d} \chi(\cdot,t) \dL x = \int_{\R^d} \chi(\cdot,0)\dL x.
			\end{align}
			
			\item \label{item:lambdaL^2}
			\emph{Square-integrable Lagrange-multiplier.}
			For any $T\in (0,T^*)$ there exists a constant $C_{\lambda}(T)<\infty$ such that
			\begin{align}\label{eq:lambdaL^2}
				\int_0^T \lambda^2(t)\dL t \leq C^2_{\lambda}(T).
			\end{align}
		\end{enumerate}
	\end{definition} 
	
	\begin{rem}
		Items~\eqref{item:V}--\eqref{item:weakevolutioneq} precisely correspond to the weak formulation in~\cite{LauxSwartz}.
		
		The optimal energy-dissipation rate in Item~\eqref{item:EDI} is the natural rate which is satisfied by any classical solution. We note that such a sharp inequality is at the heart of the definition of gradient flows~\cite{Serfaty,AGS} and has been verified for the standard mean curvature flow by Otto and the author~\cite{LauxOttoDeGiorgi}; see also~\cite{LauxLelmi} for the case of multiple phases.
		
		Finally, for the solutions constructed in~\cite{LauxSwartz, LauxSimon}, the $L^2$-bound holds with 
		\begin{align}
		C_{\lambda}^2(T) \lesssim (1+T)\Big(1+\big(E(\chi(\cdot,0))\big)^4\Big),
		\end{align}
		which follows immediately from the analogous bounds for the approximation schemes, see~\cite[Proposition 1.12]{LauxSwartz} and~\cite[Proposition 4.3]{LauxSimon}, respectively.
		
		In the case of the implicit time discretization from~\cite{MugnaiSeisSpadaro}, the $L^2$-bound on the Lagrange multiplier has been established, too, while the sharp energy-dissipation relation above is expected to hold for this scheme as well, but has not yet been derived.
	\end{rem}
	
	As in the unconstrained case of standard mean curvature flow~\cite{FHLS, Laux-LectureNotes} we define the relative entropy
	\begin{align}
	\label{eq:defErel}
	\E[\chi,\Sigma^*](t) := \int_{\R^d\times\{t\}} (1-\nu(x,t) \cdot \xi(x,t)) |\nabla \chi|
	= E[\chi(\cdot,t)]- \int_{\R^d} \chi(x,t) (\nabla \cdot \xi)(x,t) \dL x
	\end{align}
	and the volume error
	\begin{align}\label{eq:defF}
	\F[\chi,\Sigma^*](t)
	:= \int_{\R^d} |\chi(x,t)- \chi_{\Omega^*(t)}(x) | |\vartheta(x,t)| \, \dL x 
	=\int_{\R^d} (\chi(x,t)- \chi_{\Omega^*(t)}(x) ) \vartheta(x,t) \, \dL x.
	\end{align}

	Now we are in the position to formulate our second main result, which states that any calibrated flow is unique and stable in the class of distributional solutions to volume-preserving mean curvature flow.
	
	\begin{theorem}\label{thm:ws}
		Let $\Sigma^*=(\Sigma^*(t))_{t\in [0,T^*]}$ be a calibrated flow according to Definition~\ref{def:GF_Cal}. 
		Furthermore, let $\chi$ be a distributional solution of volume-preserving mean curvature flow in the sense of Definition~\ref{def:weak}. 
		Then, the relative entropy $\E(t)$ and the volume error $\F(t)$ given in~\eqref{eq:defErel} and~\eqref{eq:defF}, respectively, satisfy
		\begin{align}
			\E[\chi,\Sigma^*](t)+\F[\chi,\Sigma^*](t) \leq e^{C\sqrt{T}(1+C_\lambda(T))} \big(\E[\chi,\Sigma^*](0)+\F[\chi,\Sigma^*](0)\big) \quad \text{for a.e.\ } t\in(0,T^*).
		\end{align}
		In particular, if $\chi(x,0) = \chi_{\Omega_0}(x)$ for a.e.\ $x\in\R^d$, then
		\begin{align}\label{eq:uniqueness}
			\chi(x,t) = \chi_{\Omega^*(t)}(x) \quad \text{for a.e.\ } (x,t) \in \R^d\times(0,T^*).
		\end{align}
	\end{theorem}

	Clearly, Theorems~\ref{thm:exGFC} and~\ref{thm:ws} imply the weak-strong uniqueness of solutions to volume-preserving mean curvature flow.
	
	\begin{corollary}
		As long as a strong solution to volume-preserving mean curvature flow according to Definition~\ref{def:strong} exists, any weak solution in the sense of Definition~\ref{def:weak} with the same initial conditions has to agree with it.
	\end{corollary}

	\section{Construction of gradient-flow calibration for volume-preserving mean curvature flow}
	\label{sec:exGFC}
	
	The main purpose of this section is to prove  Theorem~\ref{thm:exGFC}. 
	Before proving this general result, it is worth mentioning the following basic example of the round sphere, for which the construction of the gradient-flow calibration is straight-forward. 
	
	\begin{expl}
		Let $\Omega^*(0) = B_R$ be a ball. 
		Then the volume-preserving mean curvature flow starting from $\Omega^*(0)$ is static: $\Omega^*(t) = B_R$. 
		Then one simply defines $\xi (x,t) := \xi(x) = \zeta(|x|-R) \frac{x}{|x|}$ and $\vartheta(x,t) := \vartheta(x) = \tau(|x|-R)$, where $\zeta$ is a cut-off around $0$ and $\tau$ is a smooth truncation of the identity. (These functions will be discussed in more detail in the following proof for the general case.) 
		Furthermore, we set $B(x,t):=0$
		and $\lambda^*(t):=\frac{d-1}{R}$.
		It is now straight-forward to see that $(\xi,B,\lambda^*)$ is a gradient flow calibration for $\Sigma^*(t) = \partial \Omega^*(t)$.
		
		The same reasoning also applies to a finite union of balls by making the localization scale in the functions $\zeta$ and $\tau$ sufficiently small.
	\end{expl}

	In the general case, the construction of $B$ is slightly more involved, and this is the heart of the matter. 
	Since the divergence-constraint~\eqref{eq:divB} is underdetermined, it is rather natural to make the ansatz $B=\nabla \varphi$ for some potential $\varphi \colon \R^d\times[0,T^*] \to \R$.
	
	As a first (overoptimistic) idea, one could hope to find an extension such that in fact $\nabla \cdot B(\cdot,t)=0$ in \emph{all} of $\R^d$. 
	This would imply that $\varphi(\cdot,t) $ would solve the following Neumann--Laplace problem
	\begin{align}
		\begin{cases}
			\Delta \varphi(\cdot,t) =0 & \text{in } \R^d \setminus \Sigma^*(t), \\
			\nu^*(\cdot,t)\cdot \nabla \varphi(\cdot,t) = V^*(\cdot,t)  &\text{on } \Sigma^*(t).
		\end{cases}
	\end{align}
	However, it turns out that this is not compatible with the regularity requirements. 
	It is not even clear that the tangential component of $B(\cdot,t)$ would be continuous across the interface $\Sigma^*(t)$
	
	Therefore, we will construct an extension $B(\cdot,t)$ which is be solenoidal only \emph{inside} $\Sigma^*(t)$, which then implies the still slightly stronger version $\nabla \cdot B(\cdot,t) =O\big(\dist(\cdot, \Omega^*(t))\big)$ of~\eqref{eq:divB}.

		\begin{proof}[Proof of Theorem~\ref{thm:exGFC}]
			
			By the assumed regularity of $\Sigma^*$, there exists $\delta=\delta(\Sigma^*)>0$ such that for all $t$, the signed distance function $s(\cdot,t)$ has the same regularity as $\Sigma^*$ in the tubular neighborhood $\mathcal{U}_\delta= \{(x,t)\in \R^d\times[0,T^*]\colon |s(x,t)| < \delta\}$ of $\Sigma^*$, see for example Ambrosio's beautiful contribution~\cite{AmbrosioDancer} or the author's lecture notes~\cite{Laux-LectureNotes}. 
			Here and throughout we use the sign convention $s(\cdot,t) <0 $ in $\Omega^*(t)$ so that $\nabla s(\cdot,t) = \nu^*(\cdot,t)$ on $\Sigma^*(t)$.
			We denote the timeslice of the neighborhood $\mathcal{U}_\delta$ by $\mathcal{U}_\delta(t) := \{x\in \R^d\colon |s(x,t)| < \delta\}$, $t\in [0,T^*]$.

			\emph{Step 1: Construction.}
			The ansatz for the extension of the normal vector field and the weight function are the ad-hoc constructions
			\begin{align*}
			\xi (x,t) := \zeta(s(x,t)) \nabla s(x,t) \quad \text{and} \quad \vartheta(x,t) := \tau(s(x,t)),
			\end{align*}
			where  $\zeta$ is a smooth cutoff function satisfying 
			$	\zeta(0)=1$ and  $\zeta(z) =0 $ for $|z| \geq \delta$, 
			and $\tau$ is a smooth and non-decreasing truncation of the identity with $\tau (z) = z$ for $|z| \leq \frac\delta2$ and $\tau(z)= \sign(z)$ for $|z| \geq \delta.$
			The parameter $\lambda^*$ is exactly given by its namesake, the Lagrange-multiplier given in~\eqref{eq:lambda*}.
			
			The construction of $B$ is slightly more involved. 
			We fix $t\in [0,T^*]$ and let $\varphi$ solve the following Neumann--Laplace problem
			\begin{align}
				\label{eq:PDEphi1}\Delta \varphi &=0 & & \text{in } \Omega^*(t), &\\
				\label{eq:BCphi1}\nu^*\cdot \nabla \varphi &= V^*(\cdot,t) & & \text{on } \Sigma^*(t).& 
			\end{align}
			The existence of this potential $\varphi$ with $\int_{\Omega^*(t)} \varphi \dL x=0$ follows from elementary elliptic theory
			thanks to the compatibility of the boundary datum with the vanishing right-hand side:
			\begin{align*}
				\int_{\Sigma^*(t)} V^*(\cdot,t)\dL \H^{d-1} 
				= \ddt |\Omega^*(t)| =0.
			\end{align*}
			By Schauder boundary regularity theory for the Neumann problem, see~\cite[Theorem 4.1]{Nardi} or \cite[Theorem 95]{LeoniLectureNotes}, we have  
			\begin{align*}
				\|\varphi\|_{C^{2,\alpha}(\overline{ \Omega^*(t)})}
				 \leq C(\Omega^*(t)) \|V^*(\cdot,t)\|_{C^{1,\alpha}(\partial\Omega^*(t))} \leq C(\Sigma^*(t)).
			\end{align*}
			
			To improve the regularity, we differentiate the equation.
			Although this is folklore (in particular away from the boundary), we provide a short argument to make sure that the estimates hold uniformly on all of $\overline{\Omega^*(t)}$; to this end, we will choose a suitable coordinate frame which is adapted to the geometry of $\partial \Omega^*(t)$.
			More precisely, in order to show that $\varphi \in C^{3,\alpha}(\overline{\Omega^*(t)})$, or in other words $\nabla \varphi \in C^{2,\alpha}(\overline{\Omega^*(t)})$, it is sufficient to prove that for any orthogonal frame $(X_1,\ldots,X_d)$ of class $C^{2,\alpha}(\overline{\Omega^*(t)})$ with $X_i \cdot \nu^* =0$ for $i=1,\ldots,d-1$ and $X_d=\nu^*$ on $\partial \Omega^*(t)$ we have $X_i \cdot \nabla \varphi \in C^{2,\alpha}(\overline{\Omega^*(t)})$ for $i=1,\ldots,d$, which we will show now.
			
			Given such a frame $(X_1,\ldots,X_d)$ and an index $i=1,\ldots,d-1$ belonging to a ``tangent'' vector field, we see that $\psi:=X_i \cdot \nabla \varphi$ solves
			\begin{align*}
				\Delta \psi &= \Delta X_i \cdot \nabla \varphi + 2\nabla X_i : \nabla^2 \varphi& & \text{in } \Omega^*(t), &\\
				\nu^*\cdot \nabla \psi &=   \nabla \varphi \cdot (\nu^*\cdot \nabla) X_i+ (X_i\cdot\nabla) V^*(\cdot,t) - \nabla \varphi \cdot (X_i\cdot \nabla) \nu^* & &\text{on } \Sigma^*(t),& 
			\end{align*}
			where we have used $\Delta \varphi = 0$ in the first line and $\nu^*\cdot \nabla \varphi = V^*(\cdot,t)$ in the second one. 
			We now recognize that the right-hand side of the PDE is of class $C^{0,\alpha}(\overline{\Omega^*(t)})$ and the Neumann boundary datum is of class $C^{1,\alpha}(\partial \Omega^*(t))$ since $V^*(\cdot,t) \in C^{2,\alpha}(\partial \Omega^*(t))$ and $\partial \Omega^*(t) \in C^{3,\alpha}$, which implies $(X_i \cdot \nabla )\nu^* \in C^{1,\alpha}(\partial \Omega^*(t))$. 
			Hence we can once more apply Schauder regularity for the Neumann--Laplace problem~\cite[Theorem 4.1]{Nardi} to assert the desired regularity~$\psi \in C^{2,\alpha}(\overline{\Omega^*(t)})$. 
			For the ``normal'' field $X_d$, we simply observe that $\psi := X_d\cdot \nabla \varphi$ solves the Dirichlet--Laplace problem
			\begin{align*}
				\Delta \psi & = \Delta X_d \cdot \nabla \varphi + 2\nabla X_d : \nabla^2 \varphi & & \text{in } \Omega^*(t), &\\
				\psi & = \nu^*\cdot \nabla \varphi = V^*(\cdot,t) & & \text{on } \Sigma^*(t).& 
			\end{align*}			
			The right-hand side of the PDE is identical to the previous case and the Dirichlet datum is of class $C^{2,\alpha}(\partial \Omega^*(t))$ by assumption. 
			Hence, we can apply standard Schauder boundary regularity theory for the Dirichlet--Laplace problem~\cite[Theorem 6.8]{GilbargTrudinger} to assert that also in this case $\psi \in C^{2,\alpha}(\overline{\Omega^*(t)})$. 
			Hence we obtain 
			\begin{align*}
				\|\varphi\|_{C^{3,\alpha}(\overline{\Omega^*(t)})}  \leq C(\Sigma^*(t)).
			\end{align*}%
			
			Now we extend $\varphi$ using a standard extension theorem, e.g.~\cite[Lemma 6.37]{GilbargTrudinger}, to a function~$\bar \varphi \in C^{3,\alpha}(\R^d)$ with the same regularity such that $\bar\varphi=\varphi $ in $\overline{\Omega^\ast(t)}$ and $\bar\varphi=0$ in $\R^d\setminus B_{2R^*}(0)$;
			in particular $\bar\varphi \in C^{2,1}(\R^d)$.
			Then we set
			\begin{align}\label{eq:regB}
				B(\cdot,t) := \nabla \bar\varphi  \in C^{1,1}(\R^d).
			\end{align}

			\emph{Step 2: Verification of all properties in Definition~\ref{def:GF_Cal}.} 
			Now we want to verify that the tuple $(\xi,B,\vartheta,\lambda^*)$ is a gradient-flow calibration according to Definition~\ref{def:GF_Cal}. 
			
			The regularity in Item~\eqref{item:reg} directly follows from the construction in \emph{Step 1}. 
			The PDE~\eqref{eq:PDEphi1} guarantees~$\nabla \cdot B(\cdot,t)=0$ in~$\Omega^*(t)$, and by the regularity of~$\bar\varphi$, we have the bound
			\begin{align}
				\nabla \cdot B(\cdot,t)= O\big(\dist(\cdot,\Omega^*(t))\big),
			\end{align}
			which in particular implies Item~\eqref{item:divB}.
			Item~\eqref{item:normal} follows directly from the construction of~$\xi$. 
			The evolution equation~\eqref{eq:extGEE} in Item~\eqref{item:GEE} is also built into the construction of~$B$, namely through the boundary condition~\eqref{eq:BCphi1}. This guarantees that on~$\Sigma^*(t)$,~\eqref{eq:extGEE} simply reduces to~\eqref{eq:VPMCF}. 
			By the Lipschitz continuity of all functions appearing on the left-hand side of~\eqref{eq:extGEE}, this implies the validity of~\eqref{eq:extGEE}.
			Item~\eqref{item:signweights} follows directly from the construction of~$\vartheta$.
			
			Now we turn to the transport equations in Item~\eqref{item:transport}. 
			Next to $B(\cdot,t)$, we will also work with the trivial extension of $B(\cdot,t)$ to the neighborhood $\mathcal{U}_\delta(t)$ which we denote by $\bar B(\cdot,t) := B(\cdot,t) \circ P_{\Sigma^*(t)}$.
			We start with the derivation of~\eqref{eq:transp_weight}. Since $\vartheta$ is a function of the signed distance function to $\Sigma^*(t)$, it holds 
			\begin{align*}
				\partial_t \vartheta + (\bar B\cdot \nabla) \vartheta =0\quad \text{in } \mathcal{U}_\delta(t),
			\end{align*}
			cf.~\cite{AmbrosioDancer, Laux-LectureNotes}, and hence 
			\begin{align*}
				\partial_t \vartheta + ( B\cdot \nabla) \vartheta = (B-\bar B)\cdot \nabla \vartheta \quad \text{in } \mathcal{U}_\delta(t)
			\end{align*}
			and the assertion follows from the Lipschitz continuity of the functions $B$ and $\vartheta$. 
			To justify the higher-order accuracy in the transport equation~\eqref{eq:transp_absxi} for the lenght of $\xi$, we write
			\begin{align}\label{eq:derive_transp_absxi}
			\big(\p_t  +(B\cdot \nabla)\big)|\xi|^2 
			=\big(\p_t  +(\bar B\cdot \nabla)\big)(\zeta^2\circ s)
			+ (B-\bar B) \cdot \nabla (\zeta^2\circ s).
			\end{align}
			The first term vanishes exactly in the neighborhood $\mathcal{U}_\delta(t)$ of $\Sigma^*(t)$. For the second one, we use the Lipschitz estimate $|B-\bar B| \leq C |s|$ and compute $|\nabla (\zeta^2\circ s)| = 2 (\zeta\circ s) |\zeta'\circ s| |\nabla s| \leq C |s|$, where we have used $\zeta'(0)=0$ and the regularity of all functions involved in the last step. Hence the right-hand side of~\eqref{eq:derive_transp_absxi} is indeed $O(s^2)$.
			The approximate transport equation for $\xi$ follows similarly: We compute
			\begin{align*}
			\p_t \xi + (B\cdot \nabla ) \xi + (\nabla B)^{\transpose} \xi 
			=& \Big(\big( \p_t+ (B\cdot \nabla)\big) (\zeta\circ s)\Big)\nabla s 
			\\&+(\zeta\circ s) \big( 	\p_t \nabla s + (B\cdot \nabla ) \nabla s + (\nabla B)^{\transpose} \nabla s\big).
			\end{align*}
			Arguing as before, the first right-hand side term is $O(s)$. For the second term,  next to $\bar B$, we also need to smuggle in $\nabla \bar B$, which will produce the leading term $f\,\xi$ on the right-hand side of~\eqref{eq:transp_xi}.
			In $\mathcal{U}_\delta(t) \supset	\supp (\zeta\circ s)$, we have
			\begin{align*}
			\big( 	\p_t \nabla s + (B\cdot \nabla ) \nabla s + (\nabla B)^{\transpose} \nabla s\big)
			=& \big( 	\p_t \nabla s + (\bar B\cdot \nabla ) \nabla s + (\nabla \bar B)^{\transpose} \nabla s\big)
			\\&+ (B-\bar B) \cdot \nabla^2 s + (\nabla B - \nabla \bar B)^{\transpose}\nabla s.
			\end{align*}
			The first right-hand side term vanishes identically in $\mathcal{U}_\delta(t)$. 
			The second term is $O(s)$, while the last term satisfies 
			\begin{align*}
			(\nabla B - \nabla \bar B)^{\transpose}\nabla s
			=& (\nabla B - (\nabla B\circ P_{\Sigma^*}) \nabla P_{\Sigma^*})^{\transpose} \nabla s
			\\=& (\nabla B -\nabla B\circ P_{\Sigma^*})^{\transpose} \nabla s
			+(I_d- \nabla P_{\Sigma^*}) (\nabla B\circ P_{\Sigma^*})^{\transpose} \nabla s.
			\end{align*}
			The first-right hand side term is $O(s)$ since $\nabla B$ is Lipschitz. Furthermore,  since $\nabla P_{\Sigma^*} = I_d - \nabla s \otimes \nabla s+O(s)$, the second term is of the form
			\begin{align*}
			\nabla s \otimes \nabla s (\nabla B \circ P_{\Sigma^*})^{\transpose}  \nabla s+O(s)
			= (\nabla s \cdot (\nabla B \circ P_{\Sigma^*})^{\transpose}  \nabla s) \nabla s+O(s).
			\end{align*}
			Recalling that the whole error term was multiplied by $\zeta\circ s$, we obtain the approximate transport equation~\eqref{eq:transp_xi} with $f:= \chara_{\mathcal{U}_\delta(t)}\nabla s \cdot (\nabla B \circ P_{\Sigma^*})  \nabla s$. 
			Note that~\eqref{eq:regB} implies $\|f(\cdot,t)\|_{L^\infty(\R^d)} \leq C(\Sigma^*)$. 
			This concludes the proof of Theorem~\ref{thm:exGFC}.
		\end{proof}
		
	\section{Relative entropy inequality and weak-strong uniqueness principle}
	\label{sec:ws}
		
	The main purpose of this section is the proof of the relative entropy inequality in Theorem~\ref{thm:ws}. 
	Let us first collect the basic coercivity properties of the relative entropy functional.
	
	\begin{lemma}\label{lem:coerc}
		The relative entropy $\E$ defined in~\eqref{eq:defErel} satisfies
		\begin{align}
			   \label{eq:tilt} \int_{\R^d\times\{t\}} \frac12 |\nu-\xi|^2 |\nabla \chi| &\leq  \E[\chi,\Sigma^*](t),
			\\ \label{eq:dist2} \int_{\R^d\times\{t\}} \vartheta^2 |\nabla \chi| &\lesssim \E[\chi,\Sigma^*](t).
		\end{align}
	\end{lemma}
	\begin{proof}
		We use the trivial identity $2(1-\xi\cdot \nu) = |\nu|^2  +|\xi|^2- 2\xi \cdot \nu + (1-|\xi|^2) = |\nu - \xi|^2 + (1+|\xi|)(1-|\xi|)$. 
		Since both terms on the right are non-negative (cf.~\eqref{eq:xishort}), the first estimate~\eqref{eq:tilt} then follows directly from the definition~\eqref{eq:defErel}, 
		and the second estimate~\eqref{eq:dist2} follows from the quantitative shortness condition~\eqref{eq:xishort} and the Lipschitz continuity of the weight function $\vartheta$.
	\end{proof}

	Now we give the proof of Theorem~\ref{thm:ws}, which partly follows the weak-strong uniqueness proof in the unconstrained case of standard mean curvature flow~\cite{FHLS}. 
	To be self-contained, we carry out the full proof here.
	Special attention will be given to the additional difficulties arising in our case with the volume-constraint.
	
	\begin{proof}[Proof of Theorem~\ref{thm:ws}]
		For notational convenience, we will suppress the dependence of the functionals on $\chi$ and $\Sigma^*$ and write $\E(t) := \E[\chi,\Sigma^*](t)$ and $\F(t) := \F[\chi,\Sigma^*](t)$.
		
		\emph{Step 1: First manipulations of relative entropy and bulk error.}
		For almost every $T\in (0,T^*)$, using the definition~\eqref{eq:defV} of~$V$ with $\zeta = \nabla \cdot \xi$, we may compute
		\begin{align*}
			\E(T) - \E(0) = E[\chi(\cdot,T)] - E[\chi(\cdot,0)]
			+ \int_{\R^d\times(0,T)} \Big( -V (\nabla \cdot \xi) - \partial_t \xi \cdot \nu \Big) |\nabla \chi|.
		\end{align*}
		Using the optimal energy-dissipation relation~\eqref{eq:EDI} and the fact that $\int_{\R^d\times(0,T)} V| \nabla \chi| =0$ (which follows from~\eqref{eq:defV} with $\zeta=1$ together with~\eqref{eq:VolPres}) to smuggle in the constant $\lambda^\ast = \lambda^\ast(t)$, we obtain
		\begin{align}\label{eq:ws_ddtE_1}
			\E(T) -\E(0) \leq 
			 \int_{\R^d\times(0,T)} \Big( - V^2   -V (\nabla \cdot \xi - \lambda^\ast) - \partial_t \xi \cdot \nu \Big) |\nabla \chi| \dL t.
		\end{align}
		We denote the ``(negative) dissipation functional'' on the right by
		\begin{align*}
			\D(t):= \int_{\R^d\times\{t\}} \Big( - V^2   -V (\nabla \cdot \xi - \lambda^\ast) - \partial_t \xi \cdot \nu \Big) |\nabla \chi|, 
		\end{align*}
		so that we have
		\begin{align*}
			\E(T) -\E(0) \leq \int_0^T \D(t)\dL t.
		\end{align*}
		
		For the bulk error $\F(t)$ defined in~\eqref{eq:defF}, using the definition~\eqref{eq:defV} of the normal velocity $V$ and the fact that $\vartheta =0$ on $\Sigma^*$, we may compute for almost every $T\in (0,T^*)$
		\begin{align*}
			\F(T) - \F(0) 
			&= \int_{\R^d\times(0,T)} \p_t \vartheta (\chi-\chi_{\Omega^*}) \dL x\dL t
			+ \int_{\R^d\times(0,T)} \vartheta V |\nabla \chi| \dL t.
		\end{align*}
		In analogy to the previous discussion for $\E(t)$, we denote the integrand on the right-hand side by
		\begin{align*}
		\widetilde \D(t) :=\int_{\R^d\times\{t\}}& \p_t \vartheta (\chi-\chi_{\Omega^*}) \dL x
		+ \int_{\R^d\times\{t\}} \vartheta V |\nabla \chi|.
		\end{align*}
		
		\emph{Step 2: Dissipation estimates. We claim that there exists a null set $\mathcal{N}\subset (0,T^*)$ and a constant $C=C(d,T^*,\Sigma^\ast)<\infty$ such that for all $ t\in (0,T^*)\setminus \mathcal{N}$ we have the estimates
		\begin{align}\label{eq:DleqE+F}
			\D(t)
			+ \frac12\int_{\R^d\times\{t\}} \big| V\nu -(B\cdot \xi ) \xi \big|^2 \,|\nabla \chi|  
			\leq C (1+|\lambda(t)|) (\E(t) +\F(t))
		\end{align}
		and 
		\begin{align}\label{eq:tildeDleqE+F+diss}
			\widetilde D(t) \leq C (\E(t)+\F(t)) + \frac12\int_{\R^d\times\{t\}} \big| V\nu -(B\cdot \xi ) \xi \big|^2 \,|\nabla \chi|.
		\end{align}}
		
		We fix $t\in (0,T^*) \setminus \mathcal{N}$, where the null set $\mathcal{N}$ is such that~\eqref{eq:weakevolutioneq} holds for all $t\in (0,T^*) \setminus \mathcal{N}$. 
		To ease notation, we omit the domain of integration $\R^d\times\{t\}$ in the following derivation of~\eqref{eq:DleqE+F} and~\eqref{eq:tildeDleqE+F+diss}.
		
		Testing the weak form~\eqref{eq:weakevolutioneq} of the evolution equation with the vector field $B$ from the gradient-flow calibration, we may rewrite $\D(t)$ as
		\begin{align*}
			\D(t) =
			 \int \Big( - V^2  -V (\nabla \cdot \xi - \lambda^*) 
			+(V-\lambda)\nu \cdot B 
			+ \nabla \cdot B - \nu \cdot \nabla B \nu - \partial_t \xi \cdot \nu \Big) |\nabla \chi|.
		\end{align*}
		After first decomposing the vector field $B$ into its ``normal'' and ``tangential'' components $B=(B\cdot \xi ) \xi+ (I_d-\xi \otimes\xi) B$, then completing the two squares (involving $V$ and $V\nu$, respectively), and adding zero to make the transport term $\p_t \xi + (B\cdot \nabla)\xi + (\nabla B)^{\transpose} \xi$ appear in the last integral, we arrive at
		\begin{align*}
			\D(t)
			+ &\frac12 \int \big( V+\nabla \cdot \xi -\lambda^*\big)^2\, |\nabla \chi| 
			+ \frac12 \int \big| V\nu -(B\cdot \xi ) \xi \big|^2 \,|\nabla \chi| 
			\\\leq& \int\frac12 \big( (\nabla \cdot \xi -\lambda^* )^2 + (B\cdot\xi)^2 |\xi|^2\big) \,|\nabla \chi| 
			+ \int \big(V\nu \cdot (I_d - \xi \otimes \xi) B -\lambda \nu \cdot B \big)\, |\nabla \chi|
			\\&+ \int \big(   \nabla \cdot B - \nu \cdot \nabla B \nu +\nu \cdot (B\cdot \nabla) \xi + \xi \cdot(\nu \cdot \nabla)B \big) \, |\nabla \chi|
			\\&-\int \nu \cdot \left( \p_t \xi +(B\cdot \nabla) \xi +(\nabla B)^{\transpose} \xi \right)  \,|\nabla \chi|.
		\end{align*}
		Now we complete another square, use $-\nu\otimes \nu +\xi \otimes \nu = -(\nu-\xi)\otimes(\nu-\xi) -\xi\otimes \nu +\xi\otimes\xi$, and also manipulate the last term to express the right-hand side as
		\begin{align} 
			\notag&\int\frac12 \big( \nabla \cdot \xi -\lambda^* + B\cdot \xi\big)^2\,|\nabla \chi| 
			+\frac12 \int \big(|\xi|^2-1\big) (B\cdot \xi)^2\,|\nabla \chi|
			- \int (\nabla \cdot \xi -\lambda^*)	B\cdot \xi \, |\nabla \chi|
			\\\notag& +\int \big(V\nu \cdot (I_d - \xi \otimes \xi) B -\lambda \nu \cdot B \big)\, |\nabla \chi|
			\\\notag& +\int (\nabla \cdot B) ( 1-\xi \cdot \nu) \, |\nabla \chi|
			+\int (\nabla \cdot B) \xi \cdot \nu |\nabla \chi|
			-\int (\nu -\xi) \cdot \nabla B (\nu - \xi) |\nabla \chi| 
			\\\notag& -\int \nu \cdot (\xi \cdot \nabla ) B|\nabla \chi| + \int \nu \cdot (B\cdot \nabla ) \xi |\nabla \chi| 
			\\ \notag&-\int  (\nu-\xi)\cdot\left( \p_t \xi +(B\cdot \nabla) \xi +(\nabla B)^{\transpose} \xi \right) \,|\nabla \chi|
			\\& \label{eq:before sym}-\int \xi \cdot \left( \p_t \xi +(B\cdot \nabla)\xi \right) \,|\nabla \chi|.
		\end{align}
		By symmetry and Gauss' theorem
		\begin{align*}
			0=\int \chi \,\nabla \cdot \big(\nabla \cdot ( B\otimes \xi - \xi \otimes B)\big) \dL x
			 = \int \nu  \cdot \big(\nabla \cdot ( B\otimes \xi - \xi \otimes B)\big)  \, |\nabla \chi|.
		\end{align*}
		Expanding the divergence in the last integral and reordering terms yields
		\begin{align*}
			\int \big((\nabla \cdot B) \xi \cdot \nu -\nu\cdot (\xi \cdot \nabla )B + \nu \cdot (B\cdot \nabla) \xi \big) |\nabla \chi| = \int (\nabla \cdot \xi) B\cdot \nu |\nabla \chi|.
		\end{align*}
		We may use this symmetry to replace three of the terms in~\eqref{eq:before sym} by the single term $(\nabla \cdot \xi) B\cdot \nu$ to write~\eqref{eq:before sym} as
		\begin{align*}
			&\int\frac12 \big( \nabla \cdot \xi -\lambda^* + B\cdot \xi\big)^2\,|\nabla \chi| 
			+\frac12 \int \big(|\xi|^2-1\big) (B\cdot \xi)^2\,|\nabla \chi|
			- \int (\nabla \cdot \xi -\lambda^*)	B\cdot \xi \, |\nabla \chi|
			\\&+\int (\nabla \cdot \xi -\lambda) B\cdot \nu \, |\nabla \chi|
			+\int V\nu \cdot (I_d - \xi \otimes \xi) B \, |\nabla \chi|
			\\& +\int (\nabla \cdot B) ( 1-\xi \cdot \nu) \, |\nabla \chi|
			-\int (\nu -\xi) \cdot \nabla B (\nu - \xi) |\nabla \chi| 
			\\&-\int (\nu-\xi) \cdot \left( \p_t \xi +(B\cdot \nabla) \xi +(\nabla B)^{\transpose} \xi \right) \,|\nabla \chi|
			\\& - \frac12 \int\left( \p_t |\xi|^2 +(B\cdot \nabla) |\xi|^2 \right) |\nabla \chi|.
		\end{align*}
		Finally, combining the third to fifth terms of the last display and using $B\cdot (\nu-\xi) -\nu\cdot(I_d-\xi \otimes \xi) B = (\nu \cdot \xi - 1) (B\cdot \xi)$, we obtain in total
		\begin{align}
			\notag \D(t) &+ \frac12 \int \big( V+\nabla \cdot \xi -\lambda^*\big)^2 |\nabla \chi| + \frac12 \int \big| V\nu -(B\cdot \xi ) \xi \big|^2 \,|\nabla \chi| 
			\\\notag\leq&\int\frac12 \big( \nabla \cdot \xi -\lambda^* + B\cdot \xi\big)^2\,|\nabla \chi| 
			+\frac12\int \big(|\xi|^2-1\big) (B\cdot \xi)^2\,|\nabla \chi|
			\\\notag&- \int (\nabla \cdot \xi -\lambda^*)	(1-\xi \cdot \nu )B\cdot\xi\, |\nabla \chi|
			+\int (\lambda^*-\lambda) B\cdot \nu \, |\nabla \chi|
			\\\notag&+\int  (V+\nabla \cdot \xi - \lambda^*)\nu \cdot (I_d - \xi \otimes \xi) B \, |\nabla \chi|
			\\\notag& +\int (\nabla \cdot B) ( 1-\xi \cdot \nu) \, |\nabla \chi|
			-\int (\nu -\xi) \cdot \nabla B (\nu - \xi) |\nabla \chi| 
			\\\notag&-\int  (\nu-\xi) \cdot \left( \p_t \xi +(B\cdot \nabla) \xi +(\nabla B)^{\transpose} \xi \right) \,|\nabla \chi|
			\\&\label{eq:ws last step} - \frac12 \int\left( \p_t |\xi|^2 +(B\cdot \nabla) |\xi|^2 \right) |\nabla \chi|.
		\end{align} 
		We claim that the right-hand side of~\eqref{eq:ws last step} is estimated by $C(1+|\lambda(t)|)(\E(t) +\F(t))$ for some $C=C(\Sigma^*)$; we argue term-by-term. 
		
		We start with the two terms which have to be handled differently than in the case of standard mean curvature flow: the fourth term $\int(\lambda^*-\lambda) B\cdot \nu |\nabla \chi|$ looks rather worrying and seems to require a stability analysis for the Lagrange-multipliers $\lambda$ and $\lambda^*$. 
		However, since by~\eqref{eq:xiext},~\eqref{eq:extGEE}, and~\eqref{eq:deflambda_cali}
		\begin{align*}
			\int (\nabla \cdot B) \chi_{\Omega^*} \dL x
			 = \int_{\Sigma^*(t)} B \cdot \nu^* \dL \H^{d-1} 
			 &= \int_{\Sigma^*(t)} B \cdot \xi \dL \H^{d-1} 
			 \\&= \int_{\Sigma^*(t)} (-\nabla \cdot \xi +\lambda^*)  \dL \H^{d-1} 
			 = 0,
		\end{align*} 
		we have
		\begin{align*}
			\int(\lambda^*-\lambda) B\cdot \nu |\nabla \chi|
			&= (\lambda^*-\lambda)  \int \chi(\nabla \cdot B) \dL x
			\\&= (\lambda^*-\lambda)  \int (\chi-\chi_{\Omega^*}) (\nabla \cdot B) \dL x 
			\\&\lesssim (1+|\lambda|) \mathcal{F}(t),
		\end{align*}
		where we have crucially used the divergence condition~\eqref{eq:divB} on $B$ and the coercivity~\eqref{eq:thetacoercive} of the weight function. 
		Note that we also used the very rough estimate $|\lambda^*-\lambda| \leq |\lambda^*|+|\lambda| $. 
		The other term for which we have to argue differently than in the case of standard mean curvature flow is the penultimate term since the transport equation for $\xi$ is not satisfied exactly on the interface, cf.~\eqref{eq:transp_xi}. 
		Nevertheless, the leading term $f\xi$ appearing on the right-hand side of~\eqref{eq:transp_xi} is almost perpendicular to $\nu-\xi$: 
		\begin{align*}
			\bigg|\int f \xi \cdot (\nu -\xi) \,|\nabla \chi|\bigg|
			\leq \big(\sup|f| \big) \int \big( (1-\xi \cdot \nu) + (1-|\xi|^2) \big) \, |\nabla \chi|,
		\end{align*}
		which is again controlled by the relative entropy using the quantitative shortness~\eqref{eq:xishort} of $\xi$. 
		Now we argue for the remaining terms on the right-hand side of~\eqref{eq:ws last step}. 
		Thanks to the approximate evolution equation~\eqref{eq:extGEE}, the integrand of the first term is $O(s^2)$, so by~\eqref{eq:thetacoercive} and~\eqref{eq:dist2}, this term is of the desired order. The second term is small, but we do not need to give an argument for this since the term is non-negative anyways. In the third term, we simply pull out the maximum of $\big|(\nabla \cdot \xi-\lambda^*)B\cdot \xi\big|$ over $B_{R^*}(0)\times[0,T^*]$ and recognize the relative entropy functional; an analogous argument works for the sixth term.
		The fifth term is handled by Young's inequality and absorption into the first dissipation term. 
		For the remaining term in Young's inequality we use $|\nu \cdot (I_d-\xi\otimes \xi)|^2= |\nu - (\xi \cdot \nu)\xi|^2 \lesssim |\nu - \xi|^2 + (1-\nu \cdot \xi)$, which gives a contribution controlled by the relative entropy thanks to~\eqref{eq:tilt}.
		Similarly, also the seventh term is controlled. 
		For the last term, we simply use the approximate transport equation~\eqref{eq:transp_absxi}.
		
		Regarding the bulk-dissipation $\widetilde \D(t)$, we use $\int (B\cdot \nabla \vartheta) (\chi-\chi_{\Omega^*}) \dL x=  \int \vartheta B\cdot \nu |\nabla \chi|- \int (\nabla \cdot B) \vartheta (\chi-\chi_{\Omega^*}) \dL x $ to make the transport operator appear:
		\begin{align*}
			\widetilde D(t)
			=&\int_{\R^d\times\{t\}}( \p_t \vartheta +B\cdot \nabla \vartheta) (\chi-\chi_{\Omega^*}) \dL x
			+\int_{\R^d\times\{t\}} (\nabla \cdot B) \vartheta (\chi-\chi_{\Omega^*}) \dL x 
			\\& +\int_{\R^d\times\{t\}} \vartheta (V-B\cdot \nu) |\nabla \chi|.
		\end{align*}
		Now we argue term-by-term to bound the right-hand side of this identity.
		Using the transport equation~\eqref{eq:transp_weight} and the coercivity~\eqref{eq:thetacoercive} of the weight function, the first term is bounded by $C\int |\vartheta| |\chi-\chi_{\Omega^*}| \dL x = C \F(t)$. 
		The second integral is bounded by $(\sup |\nabla \cdot B|) \F(t)$. 
		For the third term, we use Young's inequality and~\eqref{eq:dist2}.
		This concludes the argument for the estimates~\eqref{eq:DleqE+F} and~\eqref{eq:tildeDleqE+F+diss}.
		
		\emph{Step 3: Conclusion.}
		Plugging the estimates~\eqref{eq:DleqE+F} and~\eqref{eq:tildeDleqE+F+diss} from \emph{Step 2} into \emph{Step 1}, we obtain
		\begin{align*}
			(\E(T)+\F(T) )-(\E(0)+\F(0))   
			\leq C \int_0^T(1+|\lambda(t)|) (\E(t) +\F(t)) \dL t
		\end{align*}
		for almost every $T\in(0,T^*)$. 
		Hence, by the $L^2$-bound~\eqref{eq:lambdaL^2}, Gronwall's inequality, and Jensen's inequality we obtain
		\begin{align}
				\E(T)+\F(T) \leq e^{C \int_0^T (1+ |\lambda(t)|) \dL t} (\E(0)+\F(0) ) 
				\leq e^{C \sqrt{T}(1+C_\lambda(T))} (\E(0)+\F(0) ) 
		\end{align}
		for almost every $T\in (0,T^*)$. 
		The uniqueness statement~\eqref{eq:uniqueness} in Theorem~\ref{thm:ws} now follows from the fact that $\F(t)=0$ implies $\chi^*(\cdot,t) = \chi(\cdot,t)$ a.e.\ in $\R^d$.
	\end{proof}

\section*{Acknowledgments}
	This project has received funding from the Deutsche Forschungsgemeinschaft (DFG, German Research Foundation) under Germany's Excellence Strategy -- EXC-2047/1 -- 390685813.
	The content of this paper was developed and parts of the paper written during a visit of the author to Centro de Investigaci\'on en Matem\'atica Pura y Aplicada (CIMPA) at  Universidad de Costa Rica. 
	The author would like to thank CIMPA and its members for the hospitality and stimulating environment.
\frenchspacing
\bibliographystyle{abbrv}
\bibliography{lit}
	
  \end{document}